\documentclass[12pt]{amsproc}
\usepackage{enumerate} 
\newtheorem{theorem}{\bf Theorem}[section]
\newtheorem{lemma}[theorem]{\bf Lemma}
\newtheorem{proposition}[theorem]{\bf Proposition}
\newtheorem{corollary}[theorem]{\bf Corollary}

\newtheorem{example}[theorem]{\bf Example}

\newcommand{\pr }{\mathrm{Pr} }

\newcommand{\K }{K_0 }

\begin{document}
\title[Commuting probability]{Commuting probability for approximate subgroups of a finite group}

\author{Eloisa Detomi}
\address{Dipartimento di Matematica \lq\lq Tullio Levi-Civita\rq\rq, Universit\`a di Padova, Via Trieste 63, 35121 Padova, Italy \\
}
\email{eloisa.detomi@unipd.it}

\author{Marta Morigi}
\address{Dipartimento di Matematica, Universit\`a di Bologna\\
Piazza di Porta San Donato 5 \\ 40126 Bologna \\ Italy}
\email{marta.morigi@unibo.it}
\author{Pavel Shumyatsky}
\address{Department of Mathematics, University of Brasilia\\
Brasilia-DF \\ 70910-900 Brazil}
\email{pavel@unb.br}


\begin{abstract} 
 For subsets $X,Y$ of a finite group $G$, we write $\pr(X,Y)$ for the probability that two random elements $x\in X$ and $y\in Y$ commute. This paper addresses the relation between the structure of an approximate subgroup $A\subseteq G$ and the probabilities $\pr(A,G)$ and $\pr(A,A)$. The following results are obtained.\medskip
 
 \noindent Theorem 1.1: Let $A$ be a $K$-approximate subgroup of a finite group $G$, and let $\pr(A,G)\geq\epsilon>0$. There are two $(\epsilon, K)$-bounded positive numbers $\gamma$ and $\K$ such that $G$ contains a normal subgroup $T$ and a $\K$-approximate subgroup $B$ such that $|A\cap B|\ge \gamma\; {\mathrm{max}}\{|A|,|B|\}$ while the index $[G:T]$ and the order of the commutator subgroup $[T,\langle B \rangle]$ are $(\epsilon, K)$-bounded.
\medskip

\noindent Theorem 1.2:  Let $A$ be a $K$-approximate subgroup of a finite group $G$, and let $\pr(A,A)\geq\epsilon>0$. There are two $(\epsilon, K)$-bounded positive numbers $\gamma$ and $s$, and a subgroup $C\leq G$ such that 
  $|C \cap A^2| > \gamma |A|$ and $|C'|\leq s$. In particular, $A$ is contained in the union of at most $\gamma^{-1}K^2$ left cosets of the  subgroup $C$.\medskip
  
It is also shown that the above results admit approximate converses.
 \end{abstract}
\maketitle
\section{Introduction}
If $G$ is a finite group and $X,Y$ are subsets of $G$, we write $\pr(X,Y)$ for the probability that two random elements $x\in X$ and $y\in Y$ commute. Thus,
\[ \pr(X,Y) =\frac{ | \{ (x,y) \in X\times Y \mid xy=yx \} |}{|X|\,|Y|}.\]
The number $\pr(G,G)$ is called the commuting probability of $G$.  It is well-known that $\pr(G,G)\leq5/8$ for any nonabelian group $G$. Another important result is the theorem of P. M. Neumann \cite{neumann} which states that 
 if $G$ is a finite group and $\epsilon$ is  a positive number such that $\pr(G,G)\geq\epsilon$,  
 then $G$ has a nilpotent normal subgroup $R$ of nilpotency class at most $2$ such that both the index $[G:R]$ 
and the order of the commutator subgroup $[R,R]$ are $\epsilon$-bounded  (see also \cite{eberhard}).

Throughout the article we use the expression ``$(a,b,\dots)$-bounded" to mean that a quantity is bounded from above by a number depending only on the parameters $a,b,\dots$.

There are several recent papers studying $\pr(H,G)$, where $H$ is a subgroup of $G$  (see for example \cite{DS,Erf,nath}).
In particular, it was proved in  \cite[Proposition 1.2]{DS}  that if  $H$ is a subgroup of a finite group $G$ and $\pr(H,G)\geq\epsilon>0$, 
 then there is a normal subgroup $T\leq G$ and a subgroup $B\leq H$ such that the indices $[G:T]$ and $[H:B]$, and the order of the commutator subgroup $[T,B]$  are $\epsilon$-bounded. 

Lately there also has been a considerable interest in studying approximate subgroups of finite groups.  

Let $K$ be a positive real number.  A subset $A$ of a finite group $G$ is said to be a $K$-approximate subgroup of $G$, or simply a $K$-approximate  
group, if $A$ contains $1$ and the inverse of each of its elements, and if there exists $E \subseteq G$ with $|E| \le K$
such that $A^2\subseteq EA$.

Here and throughout, given a positive integer $j$ and a subset $X$ of a group $G$, we write $X^j$ for the set of all products $x_1\dots x_j$, where $x_i\in X$. The definition of approximate subgroups was introduced by Tao in \cite{tao}. Since then many important results on the subject have been established. In particular, Breuillard, Green and Tao essentially described the structure of finite approximate subgroups \cite{bgt}. The reader is referred to the book \cite{toint2}  for detailed information on these developments.

In the present paper we examine the relation between the structure of an approximate subgroup and the commuting probability. Using the well known analogy between approximate groups and groups, we aim at extending the above mentioned group-theoretical results to approximate subgroups. On the one side, we employ the group-theoretical machinery already available, and, on the other side, ad hoc techniques for approximate subgroups.

Proposition 1.2  in \cite{DS} says, roughly speaking, that  a subgroup $H$ of a finite group $G$ has many commuting elements with $G$ if and only if $H$ has a large subgroup which almost commutes with a large normal subgroup of $G$.  Replacing $H$ with an approximate subgroup $A$  we obtain that $A$ has many commuting elements with $G$ if and only if it is commensurate with another approximate subgroup $B$ which generates a subgroup  $H$ of $G$ with the same properties as in the aforementioned proposition. More formally, our result is as follows.

\begin{theorem}\label{main} 
Let   $A$ be a $K$-approximate subgroup of a finite group $G$, and let $\pr(A,G)\geq\epsilon>0$. 
 There are two $(\epsilon, K)$-bounded positive numbers $\gamma$ and $\K$
 such that $G$ contains a normal subgroup $T$ and a $\K$-approximate subgroup $B$ such that 
 \begin{enumerate}[(i)] 
  \item $|A\cap B|\ge \gamma\; {\mathrm{max}}\{|A|,|B|\}$, and
 \item the index $[G:T]$ and the order of the subgroup $[T,\langle B \rangle]$ are both $(\epsilon, K)$-bounded.
\end{enumerate} 
\end{theorem}

Throughout, $\langle X\rangle$ denotes the subgroup generated by a subset $X$ of a group.

Next, we study approximate subgroups $A$ such that $\pr(A,A)\geq\epsilon>0$. Roughly speaking, Neumann's theorem says that if a finite group has many commuting elements then it has a large subgroup which is almost abelian, in the sense that it has a small commutator subgroup. It turns out that if $A$ is an approximate subgroup with  many commuting elements then a big part of $A^2$ is contained in a subgroup $C$ with small commutator subgroup, and $A$ itself is contained in boundedly many cosets of $C$.

\begin{theorem}\label{main2} 
Let $A$ be a $K$-approximate subgroup of a finite group $G$ and assume that $\pr(A,A)\geq\epsilon>0$. There are two $(\epsilon, K)$-bounded positive numbers $\gamma$ and $s$, and a subgroup $C\leq G$ such that 
 \begin{enumerate}[(i)]  
  \item $|C \cap A^2| > \gamma |A|$, and
  \item the commutator subgroup of $C$ is of order at most $s$. 
\end{enumerate} 
	Moreover $A$ is contained in the union of at most $\gamma^{-1}K^2$ left cosets of the  group $C$.
\end{theorem}

Furthermore, we show that each of the above theorems admits an ``approximate"  converse.

\begin{proposition}\label{converse-main} 
Let   $A,B$ be subsets and $T$ a subgroup of a finite group $G$. Set $\gamma=|A\cap B| / |A|$,  $n=[G:T]$ and 
 $m=|[T,\langle B \rangle]|$. Then  $\pr(A,G)\geq \frac{\gamma}{nm}$. 
\end{proposition}

\begin{proposition}\label{converse2} 
Let   $A$ be a $K$-approximate subgroup of a finite group $G$, and let $C\leq G$ be a subgroup.
Set $\gamma=|C \cap A^2|/ |A|$ and $s=|C'|$.
 Then 
 \[\pr (A^2, A^2) \geq  \frac{\gamma^2}{K^4 s}.\]
\end{proposition}
Note that some properties of the commuting probability of subgroups cannot be extended verbatim to approximate subgroups  but they do hold when replacing an approximate subgroup $A$ with $A^2$ (compare Proposition \ref{lem:square} and Example \ref{example1}). 
Clearly, $A^2$ is a $K^2$-approximate subgroup, whenever $A$ is a $K$-approximate subgroup (if $A^2\subseteq EA$, then $A^4\subseteq E A^3 \subseteq E^2 A^2$).

 The next section contains some general comments on the commuting probabilities. In Section 3 we prove Theorems \ref{main} and \ref{main2}. The last section is devoted to Propositions \ref{converse-main} and \ref{converse2}.

\section{General comments on commuting probabilities} 

Note that if $X,Y$ are subsets of a finite group $G$, we have
\[   \pr(X,Y) =\frac 1 { |Y|} \sum_{y\in Y}\frac{|C_{X}(y)| }{|X| }= \frac 1 {|X|} \sum_{x\in X} \frac{|C_{Y}(x)| }{|Y| }.\]

It was proved in \cite[Lemma 2.3]{DS} that if $K$ and $N$ are  subgroups of a finite group $G$, with $N$ normal in $G$, then
 \[\pr(K, G) \le \pr \left(KN/N, G/N \right)\,\,\pr(N\cap K, N).\]
We are interested in finding an ``approximate" variant of this result. We start by looking at symmetric subsets. 
\begin{proposition}\label{lemmaq1}
Let $G$ be a finite group, $N$ a normal subgroup of $G$,  and assume that $A$  is a symmetric subset of $G$.
Then 
\[\pr(A,G)\le \frac{|A^5|}{|A|}\, \pr \left(\frac{AN}{N}, \frac{G}N\right) \, \pr(A^4\cap N,N).\]
\end{proposition}

\begin{proof}
Let $\bar G=G/N$ and $\bar A=\{hN\mid h\in A\} =\{ h_1N, \dots, h_rN\}$.  Note that $A$ is contained in the union of the subsets $h_i (N\cap A^2)$,  for $i=1, \dots, r$. Indeed, $ A \subseteq \bigcup_i h_iN$  and if $a \in h_iN$, then $h_i ^{-1}a=n$ for some $n\in N \cap A^2$ and so $a=h_in\in h_i (N\cap A^2)$.
Therefore 
\begin{eqnarray*}
|A|\,|G|\,\pr(A,G) &=&\sum_{x\in A}|C_G(x)| 
 \le \sum_{hN\in \bar A}  \left(  \sum_{x\in h(N \cap A^2)}\frac{|NC_G(x)|}{|N|}|C_N(x) | \right) \\
&\le& \sum_{hN\in \bar A}\left(\sum_{x\in h(N \cap A^2)} |C_{\bar G}(hN)|\,|C_N(x)|\right)\\
&=& \sum_{hN\in \bar A}|C_{\bar G}(hN)|\left(\sum_{x\in h(N \cap A^2)} |C_N(x)|\right)\\
&=&  \sum_{hN\in \bar A} |C_{\bar G}(hN)| \left(\sum_{y\in N} |C_{ h(N \cap A^2)}(y)|\right). 
\end{eqnarray*}
If $C_{ h(N \cap A^2)}(y)\ne \emptyset$  take $y_0\in h(N \cap A^2)\cap C_G(y)$ and observe that $ h(N \cap A^2)\subseteq y_0(N \cap A^4)$. Indeed,  $y_0=hu$ with $u\in N \cap A^2$, whence $h=y_0u^{-1}$ and   $ h (N \cap A^2)\subseteq y_0 u^{-1} (N \cap A^2) \subseteq  y_0(N \cap A^4)$.
Therefore 
\begin{eqnarray*}
C_{ h(N \cap A^2)}(y)&=& h(N \cap A^2)\cap C_G(y) \\
&\subseteq &y_0C_{N \cap A^4}(y),
\end{eqnarray*}
 whence $|C_{ h(N \cap A^2)}(y)|\le |C_{N \cap A^4}(y)|$.
It follows that
\begin{eqnarray*}
 |A|\,|G|\,\pr(A,G)     &\le&
 \left(\sum_{hN\in \bar A}|C_{\bar G}(hN) | \right) \left( \sum_{y\in N} |C_{N \cap A^4}(y)| \right)\\
&=& \Big( |\bar A|\,|\bar G| \, \pr(\bar A,\bar G) \Big)   \Big( |N|\,|N\cap A^4| \, \pr( N \cap A^4,N) \Big).
\end{eqnarray*}
Thus
\[ \pr(A,G)\le\frac{|\bar A|\,| N \cap A^4|}{|A|} \, \pr(\bar A,\bar G)\, \pr( N \cap A^4,N). \]

As $|\bar A|\,|N \cap A^4|\le |A|^5$ by \cite[Lemma 2.6.3]{toint2}, the result follows.
\end{proof}

The previous proposition says, in particular, that in a homomorphic image $\bar G$ of $G$ the commuting probability of $\bar A$ in $\bar G$ is controlled in terms of $\pr(A,G)$. The next lemma deals with the commuting probability $\pr(\bar A,\bar A)$.

\begin{proposition}
 \label{lemmaq2}
Let $G$ be a finite group, $N$ a normal subgroup of $G$,   and assume that $A$  is a symmetric subset of $G$. Then 
\[\pr(A,A)\le \frac{|A^3|\,|A^5|}{|A|^2}
\, \pr\left(\frac{AN}{N}, \frac{AN}{N}\right) \, \pr(A^4\cap N,A^2\cap N).\]
\end{proposition}

\begin{proof} The proof is very similar to the proof of Proposition \ref{lemmaq1}. Thus we will sketch it, avoiding repetitions.
Let $\bar G=G/N$ and $\bar A=\{hN\mid h\in A\} =\{ h_1N, \dots, h_rN\}$. As shown in  Proposition \ref{lemmaq1}, $A$ is contained in the union of the subsets $h_i (N\cap A^2)$,  for $i=1, \dots, r$.
Moreover, for $x \in A$, 
\[ C_A(x)  \subseteq \bigcup_{hN \in C_A(x)N/N} h C_{N\cap A^2}(x), \] 
 since, for $a, h \in C_A(x)$, the equality  $aN=hN$ implies $h^{-1}a \in N\cap A^2$. 
Therefore, as $C_A(x)N/N \subseteq  C_{\bar{A}}(xN)$, we have 
\begin{equation}
|C_A(x)| \le |C_{\bar{A}}(xN)| C_{N\cap A^2}(x)|.
\end{equation}

It follows that 
\begin{eqnarray*}
|A|\,|A|\,\pr(A,A) &=&\sum_{x\in A}|C_A(x)| 
\le  \sum_{hN\in \bar A}  \left(  \sum_{x\in h(N \cap A^2)}|C_A(x)|\right)\\
&\le& \sum_{hN\in \bar A}\left(\sum_{x\in h(N \cap A^2)} |C_{\bar A}(hN)|\,|C_{N\cap A^2}(x)|\right)\\
&=& \sum_{hN\in \bar A}|C_{\bar A}(hN)|\left(\sum_{x\in h(N \cap A^2)} |C_{N\cap A^2}(x)|\right)\\
&=&  \sum_{hN\in \bar A} |C_{\bar A}(hN)| \left(\sum_{y\in {N\cap A^2}} |C_{ h(N \cap A^2)}(y)|\right). 
\end{eqnarray*}
As in the proof of Proposition \ref{lemmaq1}, the inequality $|C_{ h(N \cap A^2)}(y)|\le |C_{N \cap A^4}(y)|$ holds for every $h\in A$.
Therefore
 \[|A|\,|A|\,\pr(A,A)     \le
 \left(\sum_{hN\in \bar A}|C_{\bar A}(hN) | \right) \left( \sum_{y\in {N\cap A^2}} |C_{N \cap A^4}(y)| \right)\]
\[= \Big( |\bar A|\,|\bar A| \, \pr(\bar A,\bar A) \Big)   \Big( |{N\cap A^2}|\,|N\cap A^4| \, \pr( N \cap A^4,{N\cap A^2}) \Big).
\]
Thus
\[ \pr(A,A)\le\frac{|\bar A|^2\,| N \cap A^4||{N\cap A^2}|}{|A|^2} \, \pr(\bar A,\bar A)\, \pr( N \cap A^4,{N\cap A^2}). \]

As $|\bar A|\,|N \cap A^4|\le |A|^5$ and $|\bar A|\,|N \cap A^2|\le |A|^3$ by \cite[Lemma 2.6.3]{toint2}, the result follows.
\end{proof}

As $|A^n|/|A| \le  K^{n-1}$ for any $K$-approximate subgroup $A$ and for every integer $n\ge 2$,  the following corollary is a straightforward consequence of the above propositions.

\begin{corollary}\label{corq1}
Let $G$ be a finite group, $N$ a normal subgroup of $G$,  and assume that $A\subseteq G$  is a $K$-approximate subgroup.
Then 
\begin{itemize}
\item $\pr(A,G)\le K^4 \, \pr(AN/N,G/N) \, \pr(A^4\cap N,N),$ 

\item $\pr(A,A)\le K^6 \, \pr(AN/N,AN/N) \, \pr(A^4\cap N,A^2\cap N).$
\end{itemize}
 In particular, \begin{itemize}
\item  $ \pr(AN/N,G/N)\ge (1/{K^4})\pr(A,G)$,
\item  $\pr(AN/N,AN/N) \ge (1/K^6)\pr(A,A).$
\end{itemize}
\end{corollary}

More generally, the above corollary holds in the case where $A$ is a symmetric set containing $1$ which has small tripling, i.e. $|A^3|\le T|A|$. To see this, observe that $|A^5|\le T^3|A|$ by \cite[Proposition 2.5.3]{toint2}. 
\medskip

If $H_1 \le H_2$ are subgroups of a finite group $G$, then 
\begin{equation}  \pr(H_1,G)\ge \pr(H_2,G) \end{equation}
 (see \cite[Theorem 3.7]{Erf}).
 
As we will show in Example  \ref{example1}, the above inequality does not hold if $H_1$, $H_2$ are $K$-approximate subgroups. 
 Indeed, if $H_1$ is a $K$-approximate subgroup contained in $H_2$,  then  $\pr (H_1,G)$ might be arbitrarily small compared to $\pr (H_2,G)$ (even if $H_2$ is a subgroup). 
 We will show however that $\pr (H_1^2,G)$ is bounded away from zero (see Proposition \ref{lem:square}). 
  In the particular case where  $H_1$  is a subgroup of $G$  and $H_2$  is a $K$-approximate subgroup containing $H_1$, we get 
 \[ \pr(H_1,G)\ge \frac{1}{K} \pr(H_2,G)\]  
 (see Corollary \ref{corH1sub}).
  
 To prove these results, we need a preliminary elementary lemma. 
 
 \begin{lemma}\label{lem:conj1}
   Assume that $G$ is a finite group, $g \in G$  and $A$ is a symmetric subset of $G.$  
   Then
   \begin{itemize}
   \item[(a)]     $|  C_A(g)|\,  |  g^A|    \le |  A^2 |$.     
   \item[(b)]    $|A|\le |C_{A^2}(g) |{|g^A|} $. 
  \end{itemize} 
  \end{lemma} 
 \begin{proof}
 Note that if $g^A=\{ g^{a_1}, \dots, g^{a_r} \}$, then the subsets $C_A(g) a_i $, for $i=1, \dots, r$, are pairwise disjoint subsets of $A^2$, so 
 \[    \left|  C_A(g) \right|  \left|  g^A \right| = \left| \bigcup_{1\le i \le r} C_A(g)  a_i \right| \le  \left|  A^2 \right|,  \]
  and (a) holds. 

To prove (b), observe that whenever $a,b \in A$ are such that $g^a=g^b$, we have $ab^{-1} \in C_{G}(g) \cap A^2= C_{A^2}(g)$. 
\end{proof}

The following result  holds for approximate subgroups.

\begin{lemma}\label{lem:conj2}
   Assume that $G$ is a finite group, $g \in G$  and $A\subseteq G$ is a $K$-approximate subgroup. Then
$|g^{A^{n}}| \le K^{n-1}|g^A|$  for every $n \ge 1$.
   
  \end{lemma} 
 \begin{proof}
 As $A^2 \subseteq EA$ where $|E| \le K$, we have $A^n \subseteq E^{n-1}A $. Since $A$ is symmetric, every element $h\in A^n $ can be written as the inverse of an element $ea \in A^n \subseteq E^{n-1}A$, with $e \in E^{n-1}$ and $a \in A$. So 
	\[ g^h=g^{(ea)^{-1}}=(g^{a^{-1}})^{e^{-1}},\] 
	and since there are at most $K^{n-1}$ elements in $E^{n-1}$, we deduce that $|g^{A^{n}}| \le K^{n-1}  |g^A|.$
  \end{proof}

  \begin{proposition}\label{lem:square}
 Let $G$ be a finite group, and let $A_1 \subseteq A_2$ be symmetric subsets of $G$ such that $|A_1^2|\le K|A_1|$ and $|A_2^2|\le K'|A_2|$. Then for any subset $B\subseteq G$ we have
  \[ \pr(A_1^2,B)\ge \frac{1}{K K'} \pr(A_2,B).\]  
 \end{proposition}
\begin{proof}
Let $g \in G$. As  $|A_1^2|\le K|A_1|$, it follows from Lemma \ref{lem:conj1} (b) that  
 \[  \frac{ |C_{A_1^2}(g) | }{|{A_1^2}|} \ge  \frac1{K |g^{A_1}|}. \]
Clearly, as $A_1 \subseteq A_2$, we have $1/{|g^{A_1}|} \ge  1/{|g^{A_2}|}$. Moreover, as $|A_2^2|\le K'|A_2|$, we deduce from Lemma \ref{lem:conj1} (a) that 
\[ \frac{1}{|g^{A_2}|} \ge  \frac{ |C_{A_2}(g) | }{ K' |{A_2}|}. \]
Therefore 
\[ 
 \frac{ |C_{A_1^2}(g) | }{|{A_1^2}|} 
 \ge  \frac1{ K |g^{A_1}|} \ge 
  \frac1{K |g^{A_2}|}  
 \ge \frac{ |C_{A_2}(g) | }{KK' |{A_2}|}. 
\] 
We conclude that 
\[ \pr(A_1^2,B)=\frac{1}{|B|} \sum_{g \in B}  \frac{ |C_{A_1^2}(g) | }{|{A_1^2}|} \ge \frac{1}{KK'|B|} \sum_{g \in B} \frac{ |C_{A_2}(g) | }{|{A_2}|}=
 \frac{1}{K K'} \pr(A_2,B), \]
as claimed. 
\end{proof} 

 \begin{corollary}\label{corH1sub} 
 Let $G$ be a finite group, and  let $H$ be a subgroup contained in a $K$-approximate subgroup $A\subseteq G$. Then for any subset $B\subseteq G$ we have
  \[ \pr(H,B)\ge \frac{1}{K} \pr(A,B).\]  
 \end{corollary}
 \begin{proof} The result a straightforward consequence of the previous lemma, taking into account that $H$ is a $1$-approximate subgroup and $H^2=H$. 
 \end{proof}
 
 \begin{example}\label{example1}
{\rm Here we show that for any positive integer $K\ge 2$
 and $\epsilon >0$ there is a finite group $G$ with a subgroup $H$ and a $K$-approximate subgroup
$A\subseteq H$ of size at least $K|H|/2^K$ such that $\pr(A,G) \le \epsilon\;\pr(H,G) $.  In particular,  ${\pr(A,G)}/{\pr(H,G)}$ cannot be bounded away from zero in terms of $|A|/|H|$ and $K$. 
   
Let $V$ be an elementary abelian $2$-group of rank $n\ge K$, and let $g_1,\dots,g_n$ be a basis of $V$. Denote by $g$ the automorphism of $V$ which cyclically permutes $g_1,\dots,g_n$. Consider the natural semi-direct product $N=V\langle g\rangle$ and let $G=N\times U$, where $U$ is a finite abelian group.
Then $Z= Z(G)=\langle g_1g_2\cdots g_n\rangle\times U$. Set $z=|Z|=2\;|U|$ and note that 
\[|G|=n 2^n |U|=n 2^{n-1} z .\] 
Let $k=K-1$ and set
\[A=  \bigcup_{1 \le i \le k} g_i Z\cup \{1\}  \quad \text{and } \quad  H= \langle A \rangle = \langle g_1,\dots,g_k \rangle Z.\] 
Note that $A$ is symmetric because the $g_i$ have order $2$, moreover $A^2\subseteq BA$, where $B=\{1,g_1,\dots, g_k\}$ has size $k+1=K$, so $A$ is a $K$-approximate subgroup.
Moreover,  as $k < n$,  we have 
\[ |A|=kz+1, \quad |H|=2^{k}z, \quad \frac{|A|}{|H|}=\frac{kz+1}{ 2^{k}z} > \frac{k}{ 2^{k}} .\]
 Since $|a^G|=n$ for every $1 \neq a \in A$, we have 
\begin{eqnarray*} 
 \pr(A,G) &=& 
 \frac 1{|A|}\sum_{a\in A}\frac{|C_G(a)|}{|G|} = 
\frac 1{|A|}\left(\sum_{1 \neq a\in A}\frac{|C_G(a)|}{|G|}+\frac{|C_G(1)|}{|G|}\right)\\
&=&\frac 1{|A|}\left(\frac{|A|-1}{n}+1\right) =
\frac 1{kz+1}\left(\frac{kz}{n}+1\right)
< \frac {kz +n}{ kzn}.
\end{eqnarray*}
As $H$ is contained in $VU$, which is an abelian subgroup of index $n$ in $G$, whenever $x \in H$ we have $[G:C_G(x)] \leq n$. So we can estimate $\pr (H,G)$ as follows: 
\begin{eqnarray*} 
 \pr(H,G) &=& \frac 1{|H|}\sum_{x\in H}\frac{|C_G(x)|}{|G|} 
 =\frac 1{|H|}\left(\sum_{x \in H \setminus Z}\frac{|C_G(x)|}{|G|}+\sum_{x \in  Z}\frac{|C_G(x)|}{|G|}\right)\\
&\ge&\frac 1{|H|}\left(\frac{|H|-z}{n}+z\right) = 
\frac 1{2^{k} z}\left(\frac{2^{k}z -z+nz}{n}\right)\\
 &>& \frac 1{2^{k} z}\left(\frac{nz}{n}\right) = \frac 1{2^{k} }.
\end{eqnarray*}
 
Therefore
 \[ \frac{\pr(A,G)}{\pr(H,G)} <  \frac{kz +n}{ kzn} \, 2^{k} = \left({\frac{1}{n} +\frac{1}{kz}}\right)  \, {2^{k}}. \] 

This can be arbitrarily small if $n$ and $|U|$ are chosen large enough. }
\end{example}

\begin{example}\label{example-Y-A}
{\rm 
In the previous example $A$ was quite small compared to the subgroup $H=\langle A \rangle$, as ${|A|}/{|H|}=({kz+1})/({ 2^{k}z}) < ({k+1})/{ 2^{k}}$. 
  Now we show that, in the notation of the  previous example, we can choose  $G$, $A$, and a  $K$-approximate subgroup $A_0\subseteq  G$ containing $A$,  such that $|A|/|A_0| > k/(k+1)$ and $\pr(A, G) \le \pr(A_0,G) $. Moreover the value of $$\pr(A,G)/\pr(A_0,G) $$ cannot be bounded away from zero in terms of $|A|/|A_0|$ and $K$.  

Let $G, A, Z$ be as in Example \ref{example1} and let $A_0=Z \cup A$. Note that $A^2\subseteq BA$ and $A_0^2\subseteq BA_0$, where  $B=\{1,g_1,\dots,g_k\}$, so both $A$ and $A_0$ are $K$-approximate subgroups. 

Note that  $|A|=kz+1$ and $|A_0|=(k+1)z$, where $z=|Z|$, and so  $|A|/|A_0|=(kz+1)/(k+1)z > k/(k+1)$.
We know from the calculations in Example \ref{example1} that
\[
 \pr(A,G) < \frac {kz +n}{kzn} =  \frac 1{n} +\frac {1}{ kz}.
\]
Moreover, as $A_0=A\cup Z$ and $|C_G(y)|/|G| \ge \frac 1 n$ for every $y\in A$ (see Example \ref{example1}),
\begin{eqnarray*} 
 \pr(A_0,G) 
 &=& \frac{1}{|A_0|}\left(\sum_{y\in Z}\frac{|C_G(y)|}{|G|}+\sum_{1 \neq y\in A}\frac{|C_G(y)|}{|G|}\right) \\
&\ge& 
 \frac{1}{ (k+1) z} \left( z + \frac{ kz}{n} \right) \\
 &>& \frac{1}{ k+1}.
\end{eqnarray*}
Therefore 
 \begin{equation}\label{alpha} \frac{\pr(A,G) }{ \pr(A_0,G)}  
  <  \left( \frac 1{n} +\frac {1}{ kz} \right) \left( k+1 \right).
  \end{equation}
If $k$ is fixed and the group $G$ is chosen with $n$ and $|U|=z/2$ arbitrary large, the righthand side of \eqref{alpha} becomes arbitrary small.
}
\end{example}

\section{Proof of the main results} 

We start with two preliminary results, the first one being purely group theoretical.

\begin{lemma}\label{lem3} Let $m\geq1$, and let $G$ be a finite group containing  a subgroup $B$ such that  $[G:C_G(x)]\leq m$ for all  $x\in B$. Then there is a normal subgroup $T\leq G$ such that the index $[G:T]$  and the order of the commutator subgroup $[T,B]$  are $m$-bounded.
\end{lemma}
\begin{proof}
Note that 
\begin{eqnarray*} 
\pr(B,G) \geq\frac 1 m. 
\end{eqnarray*}
By Proposition 1.2 of \cite{DS} there exists a  normal subgroup $R$ of $m$-bounded index in $G$ and a subgroup $U$ of $m$-bounded index in $B$ such that $[R,U]$ has $m$-bounded  order. By Remark 2.6 in \cite{DS}, the normal closure $[R,U]^G$ has $m$-bounded  order.

 Passing to the quotient over $[R, U]^G$, we can assume that $R\le C_G(U)$. 
Take $m$-boundedly many elements $b_1, \dots, b_r$ in $B$ such that  $B=\langle b_1, \dots, b_r, U\rangle$. Since $[G: C_G(b_i)] \le m$ for every $i=1,\dots,r$ and $C_G(U)$ has $m$-bounded index in $G$, the intersection $C$ of $C_G(U)$ and all $r$ subgroups $C_G(b_i)$ has $m$-bounded index in $G$. So $C$ contains a normal subgroup $T$ of $G$ of
$m$-bounded index with $[T, B]=1$. This concludes the proof.\end{proof} 

The following lemma tells us that if $A$ is an approximate subgroup and $g_1,\dots, g_s$ are elements of $G$ having few $A$-conjugates, then any element of the subgroup $\langle g_1,\dots, g_s\rangle$ has few $A$-conjugates.

\begin{lemma}\label{Aconj} Let $G$ be a finite group, $A$ a $K$-approximate subgroup of $G$, and $g_1,\dots,g_s\in G$ elements such that $|g_i^A|\le m$ for every $i=1,\dots,s$. Then there exists a $(K,m,s)$-bounded integer $u$ such that $|g^A|\le u$ for every $g\in\langle g_1,\dots,g_s  \rangle.$	
\end{lemma}
\begin{proof}
It is sufficient to show that there exists a $(K,m,s)$-bounded integer $u$ and elements $d_1,\dots,d_u\in G$ such that
\begin{equation}\label{*}A\subseteq\bigcup_{1\leq i\leq u}C_{A^{2^s}}(g_1,\dots,g_s)d_i.\end{equation}

Use induction on $s$.
Let $s=1$ and $g_1^A=\{g_1^{a_{1}},\dots,g_1^{a_{m}}\}$. If $a\in A$, then there exists $k$ such that $g_1^a=g_1^{a_{k}}$, thus $aa_{k}^{-1}\in C_{A^2}(g_1)$ and $a=(aa_{k}^{-1})a_{k}$. Hence, $$A\subseteq\bigcup_{1\leq i\leq m}C_{A^{2}}(g_1)a_{i},$$ as desired.

Now assume that the result is true for $s-1$, that is, there are $(K,m,s)$-boundedly many elements $h_1,\dots,h_v\in G$ such that
\[A\subseteq\bigcup_{1\leq i\leq v}C_{A^{2^{s-1}}}(g_1,\dots,g_{s-1})h_i.\]

Set $D=C_{A^{2^{s-1}}}(g_1,\dots,g_{s-1})$. Note that by Lemma \ref{lem:conj2} the size of the class $g_s^D$ is $(K,m,s)$-bounded. Write $$g_s^D=\{g_s^{b_1},\dots,g_s^{b_r}\}$$ for suitable ${b_1},\dots,{b_r}\in D$.

As above, it follows that   $D\subseteq \bigcup_{1\le i\le r}C_{D^2}(g_{s})b_i$. In particular, we have    

\begin{eqnarray*}
D \subseteq \bigcup_{1\le i\le r}C_{D^2}(g_s)b_i 
  \subseteq  \bigcup_{1\le i\le r}C_{A^{2^s}}(g_1,\dots,g_s)b_i.
 \end{eqnarray*}
 
 We know that 
$A\subseteq\bigcup_{1\leq i\leq v}Dh_i$, whence \eqref{*} follows with $u=rv$.
This establishes the lemma.
\end{proof}

Built on the ideas of  P. M. Neumann's theorem \cite{neumann}, we get the next proposition, which 
 contains the core of the proofs of Theorem \ref{main} and Theorem \ref{main2}. 

\begin{proposition}\label{tech} 
 Let $G$ be a finite group containing a $K_1$-approximate subgroup $H$ and a $K_2$-approximate subgroup $U$ such that $\pr(H,U) \geq\epsilon>0$. Then there exists a symmetric subset $X$ of $H$, with $1 \in X$, 
  and two positive numbers $\K$  and $m$ depending only on $K_1, K_2$ and $\epsilon$ such that 
  \begin{itemize} 
    \item $|X| \ge \frac\epsilon 2 |H|$,
  \item $X^2$ is a $\K$-approximate subgroup of $G$, 
  \item $|X^2| \le \K |X|$, 
  \item  $|y^{U}| \leq m$  for every  $y\in \langle X^2 \rangle$.
  \end{itemize}
\end{proposition}

\begin{proof}
 Set
 \[X=\{x\in H \mid  |x^U|\leq 2 K_2/\epsilon \}= \{x\in H \mid  1/|x^U|\geq \epsilon/(2K_2) \}.\]
 As $|U^2|\le K_2|U|$, from Lemma \ref{lem:conj1} (a) it follows that, for any $g \in  H \setminus X$,
\[ |C_U(g)| \le \frac{K_2|U|}{|g^U|} \le  \frac{\epsilon K_2|U|}{2 K_2}= \frac{\epsilon |U|}{2} ,\]
whence 
\begin{eqnarray*} 
\epsilon |H|\,|U| &\le  & | \{ (x,y) \in H\times U \mid xy=yx \} |=\sum_{x\in H}|C_U(x)|\\
&\le & \sum_{x \in X} |U| + \sum_{x  \in H \setminus X} \frac{\epsilon}{2}|U| \\
&\le & |X| |U| +(|H| - |X|)\frac{\epsilon}{2}|U|\\
&\le & |X| |U| +\frac{\epsilon}{2} |H|\,|U|.
\end{eqnarray*}  
Therefore $({\epsilon}/{2}) |H| \le  |X|$, that is, 
\[ |X| \ge \alpha |H| \] 
for $\alpha= {\epsilon}/{2}$. 
 Thus, 
\begin{equation*} 
|X^2| \le |H^2| \le K_1|H| \le ( K_1/\alpha) |X| 
 \end{equation*}
and also 
\[ |X^3| \le |H^3| \le K_1^2 |H| \le (K_1^2/\alpha) |X|, \] 
 hence $X$ has tripling  $K_1^2/\alpha$. 
 
 It follows that $B=X^2$ is a  $\K$-approximate subgroup where $\K$ depends on $K_1$ and $\alpha$ only  (see Proposition 2.5.5 in \cite{toint2}). 

Note that $|y^U| \leq (K_2/\alpha)^2$ for every $y \in B$.

 Let $E$ be a minimal subset of $G$ such that $ B^2 \subseteq E B$ and $|E| \le \K$. 
 By minimality of $E$, for every element $e\in E$ there are $b_1,b_2,b_3\in B$ such that $b_1b_2=eb_3$ and so every element $e\in E$  can be written as a product of at most $3$ elements of $B$. 
  Therefore $|e^U| \leq ( K_2/\alpha)^6$ for every $e \in E$. 
	
 It follows from Lemma \ref{Aconj} that there exists a $(\K,K_2, \alpha)$-bounded integer $n$ such that $|g^U|\le n$ for every $g\in \langle E\rangle.$ Note that $B^i\le \langle E\rangle B$ for every $i\ge 2$, and so there exists a $(\K,K_2, \alpha)$-bounded integer $m$ such that  $|y^U| \leq m$  for every  $y\in \langle B \rangle$. 
 As $\K$ and $\alpha$ depend only on $K_1$ and $\epsilon$, the proof is complete. 
\end{proof}

Now we are ready to proceed with the proof of our main results. For the reader's convenience we restate Theorem \ref{main} here.
\medskip 

{\bf Theorem \ref{main}.} {\it 
Let   $A$ be a $K$-approximate subgroup of a finite group $G$ such that  $\pr(A,G)\geq\epsilon>0$. 
 There are two $(\epsilon, K)$-bounded positive numbers $\gamma$ and $\K$ such that $G$ contains a normal subgroup $T$ and a $\K$-approximate subgroup $B$ such that 
  \begin{enumerate}[(i)] 
 \item $|A\cap B|\ge \gamma\; {\mathrm{max}}\{|A|,|B|\}$, and
 \item the index $[G:T]$ and the order of the subgroup $[T,\langle B \rangle]$ are both $(\epsilon, K)$-bounded.
\end{enumerate} 
}

\begin{proof}
 Apply Proposition \ref{tech} with $H=A$, $U=G$, $K_1=K$ and $K_2=1$.
 Deduce that there exists a subset $X$ of $A$ and two $(K,\epsilon)$-bounded numbers $\K$ and $m$ such that
$B=X^2$ is a $\K$-approximate subgroup while $|B\cap A|\ge |X|\ge \frac\epsilon 2 |A|$ and $|y^G| \leq m$  for every  $y\in \langle B \rangle$. 
Note also that 
\[ | B \cap A| \ge \frac{\epsilon}{2} |A| \ge  \frac{\epsilon}{2K} |A^2|\ge  \frac{\epsilon}{2K} |B|,\]
so we can take $\gamma={\epsilon}/({2K})$.
  
It follows from Lemma \ref{lem3} applied to the subgroup $\langle   B\rangle$ that  there exists a normal subgroup $ T\leq  G$ such that the index $[G:T]$  and the order of the commutator subgroup $[ T,\langle  B\rangle]$ are $m$-bounded. 
Since  $m$ is $( K, \epsilon)$-bounded, the result follows. 
\end{proof}

We now deal with Theorem \ref{main2}.  \medskip

{\bf Theorem \ref{main2}.} {\it If $A$ is a $K$-approximate subgroup of a finite group $G$ satisfying $\pr(A,A)\geq\epsilon>0$, then there are two $(\epsilon, K)$-bounded positive numbers $\gamma$ and $s$, and a subgroup $C\leq G$ such that $|C \cap A^2| > \gamma |A|$ and $|C'|\le s$. Moreover, $A$ is contained in the union of at most $\gamma^{-1}K^2$ left cosets of the  group $C$.}

\begin{proof}
 We apply Proposition \ref{tech} with $H=U=A$ and $K_1=K_2=K$ and deduce that there exists a subset $X$ of $A$ and two positive numbers $\K$ and $m$ depending only on $K$ and $\epsilon$ such that  $|X| \ge \frac\epsilon 2 |A|$, the set $B=X^2$ is a $\K$-approximate subgroup, and $|y^{A}|\leq m$ for every $y\in \langle B\rangle$.

As $|B^2|\le \K|B|$, it follows from  Lemma \ref{lem:conj1} (b) that for every $y\in\langle B\rangle$ we have
\begin{equation*} 
|C_{B}(y)|\ge \frac{|B|}{\K |y^X|} \ge
\frac{|B|}{\K |y^A|} \geq \eta|B|, 
\end{equation*}
where $\eta=1/({\K m})$. 
 Since  
 \begin{equation*} 
 |C_{B}(y)| \geq \eta|B|, 
 \end{equation*}
for every $y \in \langle B \rangle$, we have  
 \[ \pr(B,\langle B\rangle) =  \frac 1{| \langle B \rangle|}  \sum_{y\in  \langle B \rangle} \frac{|C_{B}(y)|}{ | B| } 
  \ge \eta. \]
 Now we apply Proposition \ref{tech} with $H=B$  and $U=\langle B\rangle$, where  $ K_1= \K$ and $K_2=1$, to deduce that there exists a 
  symmetric 
 subset $Y$ of $B$ and two positive numbers $ K_3$ and $n$ depending only on $\K$ and $\eta$ such that  $|Y| \ge \frac\eta 2 |B|$ and $Y^2$ is a $K_3$-approximate subgroup satisfying  $|y^{\langle B\rangle}| \leq n$  for every  $y\in \langle Y^2 \rangle$.

It follows from Theorem 1.1 in \cite{pavel} that the commutator subgroup of  the subgroup $\langle (Y^2)^{\langle B \rangle} \rangle$  has $(K_3, n)$-bounded order.  As  $\langle Y^2 \rangle =\langle Y \rangle$,  we deduce that  the order of the commutator subgroup of  $ \langle Y \rangle$ is $(K_3, n)$-bounded. 
 
 Taking into account that $Y \subseteq B=X^2 \subseteq A^2$, write
    \[  |A| \le \frac{2}{\epsilon} |X| \le   \frac{2}{\epsilon}  |X^2| \le \frac{2}{\epsilon} \left( \frac{2}{\eta} |Y| \right), \]
    and 
    \begin{equation}\label{7.1.5} |\langle Y \rangle \cap A^2| \ge | Y \cap A^2| =|Y|  \ge  \frac{\epsilon \eta} 4  |A| = \gamma |A|, 
    \end{equation}
    for $\gamma= \epsilon \eta/4$. 
    
    Finally, as      
    \[  |AY| \le |A^3| \le K^2 |A|  \le  K^2 \frac 4 { \epsilon \eta} |Y|,  \]
it follows from \cite[Lemma 2.4.4]{toint2} (Ruzsa's covering lemma) that there exists a subset $F\subseteq  A$ with $|F|\le 4K^2/(\epsilon \eta)$ such that $A \subseteq FY^2$. 
   Therefore  $A$ is contained in the union of at most $4K^2/(\epsilon \eta)$ left cosets of $\langle Y \rangle$    (note that this follows also from \eqref{7.1.5} and  Lemma 7.1.5 in \cite{toint2}).
	\end{proof}

\section{The converse statements } 

In this section we prove Proposition \ref{converse-main} and Proposition \ref{converse2}, which are roughly converse to Theorem \ref{main} and Theorem \ref{main2}, respectively.

We will also often use without mention the fact that if $S,T$ are subsets of a finite group $G$ and $g\in S$ then $|g^T|\le |[S,T]|$, because the map $g^T\to \{[g,x]\,| x\in T\}$ defined by $g^{x}\mapsto g^{-1}g^x$ is a bijection.
\medskip

{\bf Proposition \ref{converse-main}}. {\it
 Let   $A,B$ be subsets and $T$ a subgroup of a finite group $G$. Set $\gamma=|A\cap B| / |A|$,  $n=[G:T]$ and 
 $m=|[T,\langle B \rangle]|$. Then  $\pr(A,G)\geq \frac{\gamma}{nm}$. 
}

\begin{proof}
Note that
\[ |A\cap \langle B \rangle| \ge  |A\cap B| \ge \gamma |A|, \]  
so $ |A\cap \langle B \rangle|/|A| \ge \gamma$ and without loss of generality we can assume that $B$ is a subgroup.   
 Moreover,  if $g \in A \cap B$ , then
 \begin{eqnarray*} 
 [G:C_G(g)] \le [G:T]\,[T:C_T(G)]\le nm, 
	\end{eqnarray*}
 since $|[T, \langle B \rangle]|=m$.
 Therefore 
\begin{eqnarray*} 
\pr(A, G)&=&  \frac 1{|A|}  \sum_{a \in A} \frac{ |C_{G}(a) |}{  |G|}  \ge \frac 1{|A|} \sum_{g \in A \cap B} \frac{ |C_{G}(g) |}{ |G|} \\
 &\ge &  \frac{|A  \cap B|}{  nm |A| }  \ge \frac{\gamma}{nm}. 
\end{eqnarray*}
\end{proof} 
\medskip

{\bf Proposition \ref{converse2}.} {\it Let   $A$ be a $K$-approximate subgroup of a finite group $G$, and let $C\leq G$ be a subgroup. Set $\gamma=|C \cap A^2|/ |A|$ and $s=|C'|$.
 Then 
 \[\pr (A^2, A^2) \geq  \frac{\gamma^2}{K^4 s}.\]
}
\begin{proof} 
Since $|C \cap A^2| = \gamma |A|$,  
 it follows from  \cite[Lemma 7.1.5]{toint2} that $A$ is contained in the union of $n \le \gamma^{-1}K^2$ left cosets of $C$, say 
\[ A \subseteq \bigcup_{1 \le i \le n }  f_i C. \] 
As $A$ is symmetric, for every $a \in A$ we have $a^{-1} \in\bigcup_i  f_i C$, whence $a=(f_i c)^{-1}= c^{-1} f_i^{-1}$ for some $i\le n$ and $c\in C$. It follows that 
 \[ A \subseteq \bigcup_{1 \le i \le n }  C f_i^{-1}. \] 
Therefore, for every $g \in G$  we have 
\[|g^A| \le |g^{\bigcup_{1 \le i \le n }  C f_i^{-1}}| \le \sum_{i-1}^{n} |(g^C)^{f_i^{-1}}| = n |g^C|. 
\]
 When  $g \in  C \cap A^2$ we also have $|g^C| \le s$, since  $|C'|=s$,  
 hence 
\begin{equation}\label{1} 
\frac{1}{ |g^A|} \ge \frac{1}{ n s}. 
\end{equation}
Moreover, by Lemma \ref{lem:conj1} (b),   ${|g^A|} \ge |A|/{ |C_{A^2}(g) |}\ge |A^2|/(K  { |C_{A^2}(g) |} )$,  which gives 
\begin{equation}\label{2} 
\frac{ |C_{A^2}(g) |}{ |A^2|} \ge \frac{1}{ K |g^A|}.
\end{equation}
Now, by \eqref{2} and  \eqref{1}
\begin{eqnarray*} 
\pr(A^2,A^2)&=&   \frac 1{|A^2|}  \sum_{a \in A^2} \frac{ |C_{A^2}(a) |}{ |A^2|}  \ge  \frac 1{|A^2|} \sum_{g \in C \cap A^2} \frac{ |C_{A^2}(g) |}{ |A^2|} \\
&\ge &  \frac 1{|A^2|}  \sum_{g \in C \cap A^2}  \frac{1}{ K |g^A| } \\
&\ge &   \frac{|C \cap A^2|}{ K n s |A^2|}.
\end{eqnarray*}
  Moreover, as $|A^2| \le K|A|$, 
\[ |C \cap A^2|=\gamma |A| \ge \gamma  |A^2|/K. \] 
Therefore
\[ \pr(A^2,A^2) \ge 
\frac{|C \cap A^2|}{ K n s |A^2|} \ge \frac{ \gamma}{K^2 n s}.\]
 Since $n \le  \gamma^{-1}K^2$, we conclude that  $\pr(A^2,A^2) \ge{ \gamma^2}/({K^4 s}),$ as claimed. 
\end{proof} 

\noindent{\bf Acknowledgements.} The first and second authors are members of GNSAGA (INDAM), and the third author was  supported by  FAPDF and CNPq. 
The authors are grateful to the anonymous referee for  many helpful suggestions.

\end{document}